\documentclass[12pt, twoside, leqno]{article}

\usepackage{amsmath,amsthm}
\usepackage{amssymb}

\usepackage{enumerate}

\usepackage{graphicx}

\newtheorem{thm}{Theorem}[section]
\newtheorem{cor}[thm]{Corollary}
\newtheorem{lem}[thm]{Lemma}
\newtheorem{prop}[thm]{Proposition}

\theoremstyle{definition}

\numberwithin{equation}{section}

\begin{document}


\baselineskip=17pt


\title{On a Mazur problem from ``Scottish Book'' concerning second partial derivatives}

\author{Volodymyr Mykhaylyuk\\
Department of Applied Mathematics\\
Chernivtsi National University,\\
58-012 Chernivtsi, Ukraine\\
E-mail: vmykhaylyuk@ukr.net
\and
Anatolij Plichko\\
Institute of Mathematics\\
Cracow University of Technology\\
31-155 Cracow, Poland\\
E-mail: aplichko@pk.edu.pl}

\date{}

\maketitle


\renewcommand{\thefootnote}{}



\renewcommand{\thefootnote}{\arabic{footnote}}
\setcounter{footnote}{0}


\begin{abstract}
We comment on a Mazur problem from ``Scottish Book'' concerning second partial derivatives.
It is proved that, if a function $f(x,y)$ of real variables defined on a rectangle has
continuous derivative with respect to $y$ and for almost all $y$ the function
$\,F_y(x):=f'_y(x,y)$ has finite variation, then almost everywhere on the rectangle there
exists the partial derivative $f''_{yx}$. We construct a separately twice differentiable
function, whose partial derivative $f'_x$ is discontinuous with respect to the second
variable on a set of positive measure. This solves in the negative the Mazur problem.
\end{abstract}

\section{Introduction}
By classical Banach's result, for an arbitrary (Lebesgue) measurable function
$f:\mathbb R\to\mathbb R$ the set $D$ of differentiability points is measurable
and the derivative $f'$ is measurable on $D$. Haslam-Jones \cite{H-D} generalized
this result for functions of several variables. More exactly, he established that
the set $D$ of differentiability points of a measurable function
$f:\mathbb R^n\to\mathbb R$ is measurable and each of its partial derivative
$f'_{x_i}$ is measurable on $D$.

Investigation of the existence and measurability of partial derivatives was
continued in \cite{Ser}, \cite{MM}. In particular, in \cite{Ser} it was proved
that for a measurable function $f(x_1,\dots, x_n)$, defined on a rectangle $P$,
which is monotone with respect to $i$th variable on almost all segments,
parallel to $i$th axe, $f'_{x_i}$ exists almost everywhere (a.e.). In \cite{MM} it
was proved that for a measurable function $f(x_1,\dots, x_n)$ the set
of existence points for the partial derivative $f'_{x_i}$ is measurable, under
a weaker assumption. Moreover, Serrin \cite{Ser} has constructed a
measurable function $f$ on $[0,1]^2$ which is a.e. differentiable on each horizontal segment as a
function of one variable, but for which the set of the existence of partial derivative
with respect to the first variable is non-measurable.

In the well known ``Scottish Book'' \cite{Mauldin} S.~Mazur posed the
following question (VII.1935, Problem 66):

{\it The real function $z=f(x,y)$ of real variables $x,y$ possesses the
1st partial derivatives $f'_x$, $f'_y$ and the pure second partial derivatives
$f''_{xx}$, $f''_{yy}$. Do there exist then almost everywhere the mixed 2nd partial
derivatives $f''_{xy}$, $f''_{yx}$? According to a remark by p. Schauder\footnote{We
do not know, what ``p`` denotes, ''pan''=``mister'' or ``professor''.}, this
theorem is true with the following additional assumptions: The derivatives
$f'_x$, $f'_y$ are absolutely continuous in the sense of Tonelli, and the derivatives
$f''_{xx}$, $f''_{yy}$ are square integrable. An analogous question for
$n$ variables.}

Mazur's problem has some interest for differential equations in partial
derivatives. If, for example, one considers the equation
$f''_{xx}+f''_{yy}=g,$ there appears a natural question on differentiability
properties of its solution (see e.g. \cite{MatiychukEidelman}).
The Mazur problem is a part of a general problem on the relations between various
partial derivatives in PDE and in the theory of function spaces connected to
derivation (see. e.g. \cite{PelczynskiWojciechowski}, \cite{Triebel}). Some results
of these theories are formulated for classes of functions in Sobolev spaces, so it
is not obvious that they bare valid for individual functions.

From the context of Problem 66 one can suppose that Mazur knew (suspected)
that the mixed derivatives need not exist {\it everywhere}. It may be a
surprise, but only in 1958 Mityagin published an example \cite{Mitiagin},
which shows that the existence and continuity of the second pure derivatives
does not imply the existence of mixed derivatives everywhere. More exactly,
he provided a function $f(x,y)$ continuous in a circle with
center at zero for which there are continuous derivatives $f''_{xx}$ and
 $f''_{yy}$, but $f''_{xy}(0,0)$ does not exist.

Bugrov \cite{Bugrov} improves this result by showing that in a square
there exists a harmonic functions $f(x,y)$ (i.e. such that
$f''_{xx}=-f''_{yy}$) whose mixed partial derivatives are unbounded.
The existence of mixed partial derivatives by means of Fourier
series was investigated in the fundamental memoir \cite{Bernstein0}.
Bernstein obtained the following results.

\begin{thm} \label{th:1.1} {\rm \cite[Th.~79]{Bernstein0}}. Let
$f(x,y)$ be a $2\pi$-periodic in both variables function, whose partial
derivatives $f^{(k)}_{x^k}$ and  $f^{(k)}_{y^k}$ are developed into double
trigonometric Fourier series whose sum of the modulus of coefficients
are not greater than $c$. Then there are all mixed derivatives of
$f$ of order $k$, which are developed into double
trigonometric Fourier series whose sums of the modulus of coefficients
are not greater than $2c$.
\end{thm}

\begin{thm} \label{th:1.2} {\rm \cite[Th.~81]{Bernstein0}}. Let
$f(x,y)$ be a $2\pi$-periodic in both variables function, whose partial
derivatives $f^{(k)}_{x^k}$ and  $f^{(k)}_{y^k}$ satisfy the H\"older condition
with index $\alpha$. Then $f$ has for any $\alpha_1<\alpha$ all mixed derivatives
of order $k$, which satisfy the H\"older condition with
index $\alpha_1$.
\end{thm}

\begin{thm} \label{th:1.3} {\rm \cite[Th.~80]{Bernstein0}}. Let
$f(x,y)$ be a $2\pi$-periodic in both variables function which has all
second partial derivatives, moreover let
$$\int_{0}^{2\pi}\int_{0}^{2\pi}(f''_{xx})^2 dxdy\leq c\;\;\;
and\;\;\;\int_{0}^{2\pi}\int_{0}^{2\pi}(f''_{yy})^2 dxdy\leq c.$$
Then $$\int_{0}^{2\pi}\int_{0}^{2\pi}(f''_{xy})^2 dxdy\leq c.\,
\footnote{The original form of the theorem is different. However,
an analysis of Bernstein's proof shows that he proved, in fact, Theorem
\ref{th:1.3}.}$$
\end{thm}

On the other hand, in the frame of the descriptive functions theory
Tolstov \cite{Tolstov} has proved the following statements.

\begin{thm}\label{th:1.4}
If a function $f(x,y)$ is separately continuous on a rectangle
and has the derivative $f''_{xx}$ everywhere, then $f''_{xx}$
is of the first Baire class.
\end{thm}

\begin{thm}\label{th:1.5}
Let $f(x,y)$ be defined on a rectangle $P$, and $f'_{x}$,
$\,f'_{y}$ have all finite derivative numbers with respect to each variable
on some subset $E\subseteq P$ of positive measure. Then a.e. on $E$ there
exist and are equal the mixed derivatives $f''_{xy}$ and $f''_{yx}$.
\end{thm}

Theorem \ref{th:1.5} implies, in particular, that if $f(x,y)$ has a.e.
all second  partial derivatives then the mixed derivatives are equal a.e.
Moreover, Tolstov has constructed \cite{Tolstov1} a function of two variables
having jointly continuous first partial derivatives and mixed second partial
derivatives  which are different on a set of positive measure.

In this paper we show that Schauder's remark is valid under significantly
weaker additional assumptions. More exactly, we prove that the Mazur
problem has a positive answer if $f'_{x}$ and $f'_{y}$ have finite variations in the
Tonelli sense. As a byproduct, we obtain a new result on measurability of
existence points for a partial derivative (Proposition \ref{pr:2.1}). Finally,
we solve the Mazur problem in the negative by constructing a separately twice
differentiable function $f$ such that $f'_x$ is discontinuous with respect to $y$
at all points of a set of positive measure.

\section{Tonelli variation and mixed derivatives}

Given a function $f:[a,b]\times[c,d]\to\mathbb R$ and $x\in[a,b]$,
we denote by $V_1(x)$ the variation of the function $f^x:[c,d]\to\mathbb R\,$,
$\,f^x(y):=f(x,y)$, and given $y\in[c,d]$ we denote by $V_2(y)$ the variation of
$f_y:[a,b]\to\mathbb R\,$, $\,f_y(x):=f(x,y)$ (these variations may be infinite).
Note that $V_1(x)$ is lower semicontinuous if $f$ is continuous
with respect to $x$ and similarly for $V_2(y)$. A function $f$ is
of {\it Tonelli bounded variation} \cite[p.~169]{S} if $\int_a^b V_1(x)dx <\infty$
and $\int_c^d V_2(y)dy<\infty$. All integrals we consider are Lebesgue integrals.

\begin{prop} \label{pr:2.1} Let $f:\mathbb R^2\to\mathbb R$ be continuous with
respect to $y$. Then the set $E$ of all points $(x,y)\in\mathbb R^2$ at which there
exists $f'_x$, has type $F_{\sigma\delta}$.
\end{prop}

\begin{proof}
Given $m,n \in\mathbb N$, denote by $A_{m,n}$ the set of
all points $(x,y)\in\mathbb R^2$ such that for all
$u,v\in (x-\frac{1}{n},\,x+\frac{1}{n})$
$$|f(u,y)-f(v,y)|\le\frac{1}{m}\,,$$
and by $B_{m,n}$ the set of points $(x,y)\in\mathbb R^2$ such that for all
$u,u'\in (x,\,x+\frac{1}{n})$ and $v,v'\in (x-\frac{1}{n},\,x)$  $$\left|\frac{f(u,y)-f(v,y)}{u-v}-\frac{f(u',y)-f(v',y)}{u'-v'}\right|\le\frac{1}{m}\,.$$
All the sets $A_{m,n}$, $B_{m,n}$ are closed, so the sets
$$A=\bigcap\nolimits_{m\in\mathbb N}\bigcup\nolimits_{n\in\mathbb N}A_{m,n}\;\;\;\;
\mbox{and}\;\;\;\;
B=\bigcap\nolimits_{m\in\mathbb N}\bigcup\nolimits_{n\in\mathbb N} B_{m,n}$$
have type $F_{\sigma\delta}$. It remains to note that $A$ is the set of all
continuity points of $f$ with respect to $x$, $\,B$ is the
set of all points $(x,y)$ for which there exists the finite
limit $$\lim\limits_{(u,v)\to(x+0,x-0)}\frac{f(u,y)-f(v,y)}{u-v}\,,$$ and
$E=A\cap B$.
\end{proof}

Note that Proposition \ref{pr:2.1} does not follow from mentioned in Introduction
papers \cite{Ser}, \cite{MM}. This statement looks to be new, even for functions of
one variable (i.e. when $f$ is constant with respect to $y$). Similar results
for continuous functions of one variable one can find in \cite[p.~309]{Hau} or
\cite[p.~228]{Bru}.

\begin{prop} \label{pr:2.2} Let $P=[a,b]\times[c,d]$\,, $\,f:P\to\mathbb R$
be continuous with respect to $y$ and for almost all
$y\in[c,d]$ the function $f_y(x):=f(x,y)$ have finite variation. Then
a.e. on $P$ there exists the partial derivative $f'_x$.
\end{prop}

\begin{proof} By Proposition \ref{pr:2.1}, the set
$E=\{(x,y)\in P: \exists\,f'_x(x,y)\}$
is measurable. Hence, $F=P\setminus E$ is also measurable.

By the assumptions of Proposition \ref{pr:2.2}, we can choose a subset $A\subseteq [c,d]$
with Lebesgue measure $\mu(A)=d-c$ such that each function $f_y$,
$\,y\in A$, has finite variation on $[a,b]$. It is well known (see e.g.
\cite[Ch.~VIII, \S2, Th.~4]{Nat}) that monotone functions (hence, functions of finite
variation) have derivative a.e. So, $\mu(F\cap([a,b]\times\{y\}))=0$ for each
$y\in A$. Now, by the Foubini theorem (or by \cite[Ch.~XI, \S5, Th.~1]{Nat}), $\mu(F)=0$.
\end{proof}

The next corollary shows that in the remark of p.~Schauder the assumption
of square integrability of $f''_{xx}$ and $f''_{yy}$ is superfluous.

\begin{cor} \label{cor:2.3}
Let $P=[a,b]\times[c,d]$\,, $\,f:P\to\mathbb R$
be a continuously differentiable with respect to $y$ function, and for
almost all $y\in[c,d]$ the function $\,F_y(x):=f'_y(x,y)$ have finite
variation $($e.g. let $f'_y$ have finite Tonelli variation$)$. Then a.e.
on $P$ there exists the partial derivative $f''_{yx}$.
\end{cor}

\begin{proof}
We can use Proposition \ref{pr:2.2} taking $f'_y$ instead of $f$.
\end{proof}

\section{Example}

For a real valued function $f$, denote $\mathrm{supp}f=\{x\in \mathbb{R}:f(x)\ne 0\}$.

\begin{lem} \label{lem:5.1} Let $I_n=(a_n,b_n)$ be pairwise
disjoint intervals and $\psi_n:\mathbb{R}\to\mathbb R$ be
differentiable functions such that $\,\mathrm{supp}\,\psi_n\subset I_n$, $n=1,2,\dots$,
and $\sup_\mathbb{R}\psi'_n(x)\to 0$ as $n\to\infty$.

Then the function $g(x)=\sum_{n=1}^{\infty}\psi_n(x)$
is differentiable, moreover, $$g'(x)=\sum\nolimits_{n=1}^{\infty}\psi'_n(x).$$
\end{lem}

The lemma easily follows from the theorem on series differentiability.

The next theorem gives a negative answer to Mazur's problem.

\begin{thm} \label{th:5.2}
There exists a separately twice differentiable function
$f:[0,1]^2\to\mathbb R$ and a measurable subset $E\subset [0,1]^2\,$, $\,\mu(E)>0$,
such that $f'_x$ is discontinuous with respect to $y$ at all points of $E$, in
particular $\;f''_{xy}$ does not exist on $E$.
\end{thm}

\begin{proof}
Let $B\subset [0,1]$ be a closed set of positive measure without isolated points
whose complement $[0,1]\setminus B$ is dense in $[0,1]$. Take intervals $I_n=(a_n,b_n)$ such
that $[0,1]\setminus B=\bigsqcup\limits_{n=1}^{\infty}I_n$. Let $\psi:\mathbb R\to\mathbb{R}^+$
be an arbitrary twice continuously differentiable function with
$\mathrm{supp}\,\psi(y)=(0,1)$ and $$\psi_n(y):=\psi\left(\frac{y-a_n}{b_n-a_n}\right)\,,
\;\;n=1,2,\dots$$

Take $\varepsilon_n, \delta_n>0$ so that
\begin{equation}\label{eq:5.1}
\lim\limits_{n\to\infty}\frac{\varepsilon_n}{(b_n-a_n)^2}=0\;\;\;\mathrm{and}
\end{equation}
\begin{equation}\label{eq:5.2}
\sum\nolimits_{n=1}^\infty\delta_n<\infty.
\end{equation}

Choose twice differentiable functions
$\varphi_n:[0,1]\to [0,\varepsilon_n]$ such that
\begin{equation}\label{eq:5.3}
\mu(A_n)>1-\delta_n,\;\;n=1,2,\dots,
\end{equation}
where $A_n=\{x\in[0,1]:|\varphi'_n(x)|\ge 1\}$.

Let us consider the function $f:[0,1]^2\to\mathbb{R}$,
$$f(x,y)=\sum\nolimits_{n=1}^{\infty}\varphi_n(x)\psi_n(y).$$

It easy to see that $f''_{xx}$ exists. Moreover, by (\ref{eq:5.1}) and Lemma
\ref{lem:5.1}, $f''_{yy}$ exists. Put
$$A=\bigcup_{m=1}^\infty \bigcap_{n\geq m}A_n.$$
Then, by (\ref{eq:5.3}) and (\ref{eq:5.2}), $\mu(A)=1$. We show that $f'_x$
is discontinuous with respect to $y$ at every point of the set
$E=A\times B$. Fix $(x_0,y_0)\in E$ and $\delta>0$. Choose $m$ so
that $x_0\in \bigcap_{n\geq m}A_n$. Since $B$ has no
isolated points, there exists $k>m$ such that
$|y-y_0|<\delta$ for all $y\in I_k$. Take $y_k\in I_k$ so that
$\psi_k(y_k)=\max\limits_\mathbb{R}\psi(y)$. Now we have $|y_k-y_0|<\delta\,$,
$\,x_0\in A_k$ and
$$\left|f'_x(x_0,y_k)-f'_x(x_0,y_0)\right|=\left|\varphi'_k(x_0)\right|\psi_k(y_k)\geq
\max\limits_\mathbb{R}\psi(y).$$
\end{proof}

Note that in the constructed example the partial derivative $f''_{yy}$ is bounded
and $f''_{xx}$ is not absolutely integrable.

\subsection*{Acknowledgements}
This research was partly supported by NSF (grant no. XXXX).

\end{document}